\documentclass[12pt,reqno]{amsart}

\usepackage{amsfonts,amsmath,amsthm}
\usepackage{amssymb,epsfig}
\usepackage{cases}
\usepackage[usenames]{color}
\usepackage{enumerate} 

\usepackage{color} 
\definecolor{vert}{rgb}{0,0.6,0}

\theoremstyle{plain}
\newtheorem{thm}{Theorem}[section]

\newtheorem{defn}{Definition}

\newtheorem{ex}{Example}
\newtheorem{lem}[thm]{Lemma}
\newtheorem{cor}[thm]{Corollary}
\newtheorem{prop}[thm]{Proposition}
\theoremstyle{remark}
\newtheorem{rem}{\bf{Remark}}
\numberwithin{equation}{section}

\newcommand{\C}{\mathbb{C}}
\newcommand{\E}{\mathbb{E}}

\newcommand{\N}{\mathbb{N}}
\newcommand{\bP}{\mathbb{P}}
\newcommand{\R}{\mathbb{R}}
\newcommand{\bS}{\mathbb{S}}

\newcommand{\USC}{{\rm USC\,}}
\newcommand{\LSC}{{\rm LSC\,}}
\newcommand{\Li}{L^{\infty}}
\newcommand{\Lip}{{\rm Lip\,}}

\newcommand{\bU}{\partial U}
\newcommand{\cU}{\overline U}

\newcommand{\al}{\alpha}

\newcommand{\del}{\delta}
\newcommand{\ep}{\varepsilon}

\newcommand{\lam}{\lambda}

\newcommand{\om}{\omega}

\newcommand{\Del}{\Delta}

\newcommand{\ol}{\overline}
\newcommand{\ul}{\underline}
\newcommand{\pl}{\partial}

\newcommand{\dist}{{\rm dist}\,}

\newcommand{\tr}{{\rm tr}\,}

\begin{document}
\title[Weakly coupled systems of the infinity Laplace equations]
{Weakly coupled systems of \\
the infinity Laplace equations}

\author{H. Mitake}
\address[H. Mitake]{
Institute for Sustainable Sciences and Development, 
Hiroshima University 1-4-1 Kagamiyama, Higashi-Hiroshima-shi 739-8527, Japan}
\email{hiroyoshi-mitake@hiroshima-u.ac.jp}

\author{H. V. Tran}
\address[H. V. Tran]{Department of Mathematics, 
The University of Chicago, 5734 S. University Avenue Chicago, Illinois 60637, USA}
\email{hung@math.uchicago.edu}

\date{\today}

\thanks{
The work of HM was partially supported by JST program to disseminate tenure tracking system, and JSPS KAKENHI \#24840042.  
}

\maketitle

\begin{abstract}
We derive the weakly coupled systems of the infinity Laplace equations via a tug-of-war game introduced by Peres, Schramm, Sheffield, and Wilson (2009). We establish existence, uniqueness results of the solutions, and introduce a new notion of ``generalized cones" for systems. By using ``generalized cones" we analyze blow-up limits of solutions. 
\end{abstract}


\tableofcontents


\section{Introduction}
In this paper, we consider the Dirichlet problem for  
the weakly coupled systems of the infinity Laplace equations: 
\begin{equation}\label{INF}
\begin{cases}
-\Del_\infty u_i+\displaystyle\sum_{j=1}^{m}c_{ij}(u_i-u_j)=0 
&\text{in}\ U \ \text{for} \ i=1,\dots,m\\
u_i=g_i 
&\text{on} \ \partial U \ \text{for} \ i=1,\dots,m,
\end{cases}
\end{equation}
where 
$U$ is a bounded domain with a smooth boundary in $\R^n$, and
$(c_{ij})_{i,j=1}^{m}$ is a given constant matrix which describes 
the generator of an irreducible continuous-time Markov chain with 
$m$ states, and $g_i\in C(\bU)$ are given functions for $i=1,\ldots,m$.  
We give the precise assumption on $(c_{ij})$ in Section \ref{sec:derivation}.
Here $u_i$ are unknown functions and 
the operator $\Del_\infty$ is the so-called \textit{game infinity 
Laplacian}, i.e., for a smooth function $f$, 
\[
\Del_\infty f
:=
\frac{\tr\big(Df\otimes Df D^2f\big)}{|Df|^2}
=\dfrac{\sum_{i,j=1}^nf_{x_i}f_{x_j}f_{x_ix_j}}{|Df|^2}. 
\]
Throughout the paper, we write 
$f_{x_i}:=\pl f/\pl x_i$, $f_{x_ix_j}:=\pl^{2} f/\pl x_i\pl x_j$, 
and denote the gradient and the Hessian 
matrix of $f$ by $Df$ and $D^2f$, respectively.

The study of the infinity Laplacian began with pioneer works by Aronsson \cite{A1, A2} to understand a so-called \textit{absolutely minimizing Lipschitz} function. 
More precisely, the equation arises in the $\Li$ calculus of variations as the Euler--Lagrange equation for properly interpreted minimizers of all of energy functionals $u\mapsto\|Du\|_{\Li(V)}$ for all open sets $V \subseteq U$. 
 Aronsson achieved existence results and pointed out that we cannot expect the classical solutions in general. However, he could not prove uniqueness and stability results.

It turned out that the theory of viscosity solution is an appropriate instrument for the study of infinity Laplacian. 
Jensen \cite{J} gave fundamental results on 
the comparison principle and hence uniqueness of the single infinity Laplace equation in the viscosity solution sense, and generated considerable interest in the theory.
Nowadays, there are a great number of works related to the infinity Laplace equation. 
We deal only with viscosity solutions in this paper, and therefore the term ``viscosity" may be omitted henceforth.

Peres, Schramm, Sheffield, and Wilson \cite{PSSW} showed that the infinity Laplace equation also arises in the study of certain two-player, zero-sum stochastic games. They introduced a random-turn game called \textit{$\ep$-tug-of-war}, in which two players try to move a token in an open set $U$ toward a favorable spot on the boundary $\pl U$ corresponding to a given payoff function $g$ on $\partial U$. 
Inspired by this work, we derive the system of the infinity Laplace equation 
\eqref{INF}, which will be described in more detail in Section \ref{sec:derivation}. 
In our setting, we consider an $\ep$-tug-of-war 
 game with $m$ modes $\{1,\dots,m\}$ and
$m$ corresponding the number of payoffs $\{g_1,\dots,g_m\}$ on $\bU$.  
The mode changes  when the players move the token each time,
 which is controlled by a piecewise-deterministic Markov
processes \eqref{markov} in Section \ref{sec:derivation}. 
Our natural interest is an equilibrium value of the expectation of such payoffs in the game
as $\ep \to 0$. We can naturally derive the weakly coupled system \eqref{INF},
and achieve the existence of solutions via this procedure. 
We also can prove the comparison principle quite straightforwardly
by using an analogous argument to that of  Barles, Busca \cite{BB}, 
which implies the uniqueness of solutions.

Our main contribution of this paper is the introduction of the ``\textit{generalized cones}" for \eqref{INF} and some detailed analysis by using the new ``cones". 
One of the key tools to analyze the infinity Laplacian is
the \textit{comparison with cones} principle, which was first introduced by Crandall, Evans and Gariepy \cite{CEG}.  
This gives not only a characterization of solutions but also key ingredients to prove the regularity of solutions (see \cite{Sa,ESa,ESm}). 
See \cite{C} for more details.  
By considering spherically symmetric solutions of \eqref{INF}, 
we derive a class of particular solutions, which give ``generalized cones" for \eqref{INF} as described in Section \ref{sec:cone}. 
Then we establish the result, Theorem \ref{thm:cone}, on the comparison with 
``generalized cones" principle. 

By using Theorem \ref{thm:cone}, we show that all blow-up limits of solutions up to passing
some subsequences are affine. This was proved in \cite{CEG} for single equation. 
It turns out that here we need to catch up a key ingredient, Lemma \ref{lem:Ll}, for systems, which does not appear in the context of the single equation. 
It is still hard to explain this Lemma in a clear and intuitive way, but it is enough informative to 
achieve the claim, Theorem \ref{thm:blow-up}.
See Remark \ref{rem:ball-sphere} for some interpretation of Lemma \ref{lem:Ll}.

We emphasize that the ``generalized cones" and their behaviors are widely different
from these of single equation case (see Remarks  \ref{rem:entire}, \ref{rem:ball-sphere}, and Example \ref{ex1} for some details.)
We finally give important remarks on \eqref{INF} in Section \ref{sec:interpretation}. 
Some of the explanation is just purely heuristic, but it does 
show some further complicated and interesting
 structure of this system.

We thank Bob Jensen for his fruitful discussions.

\section{Derivation of weakly coupled systems by a tug-of-war game}\label{sec:derivation}
Let $U\subset\R^n$ be a bounded domain with smooth boundary, which is the place where the game is played by two persons, player I and player II.
Suppose that there are  $m$ modes: mode $1$,\ \dots, mode $m$, and $m$ corresponding the number of given functions  $g_i\in C(\bU)$ for 
$i=1,\dots,m$. We call $g_i$ the payoff function on the boundary of $U$ corresponding to mode $i$ for $1\leq i \leq m$.
We consider the following two-player, zero-sum game.  

Fix a number $\ep>0$, a token $x_0:=x\in U$, and a mode $m_0:=i\in\{1,\dots,m\}$.  
Suppose that both players start the game at position $x_0=x$ and mode $m_0=i$,
and have the same position and mode all the time. 
At each time step $t_k:=\ep^2 k$ for $k\in \N$, the players toss a fair coin and the winner of the toss is allowed to choose a next token $x_k \in \ol{B}(x_{k-1},\ep)\cap\cU$, and the mode is switched from $m_{k-1}$ to mode $m_k=j$ for any $j\in \{1,\dots,m\}$ with the probability which is determined by a piecewise-deterministic Markov process introduced by Davis \cite{D}. 
The change from modes to modes with the starting point $m_0=i$ is determined by a continuous-time Markov chain on $[0,\infty)$: $\nu(0)=i$, and for  $\Del s >0$,
\begin{equation}\label{markov}
\bP\big(\nu(s+\Del s)=j\mid 
\nu(s)=i\big)=\frac{c_{ij}}{2}\Del s+o(\Del s) \quad 
\textrm{as} \ \Del s\to0 \ \textrm{for} \ i\not=j, 
\end{equation}
where 
$c_{ij}$ are given constants satisfying 
\begin{equation}\label{gen-cij}
\ c_{ij}> 0 \ \text{for}\ i \ne j,\quad \text{and} \quad
\sum_{j=1}^{m} c_{ij}=0, 
\end{equation}
and $o:[0,\infty)\to[0,\infty)$ is a function satisfying $o(r)/r\to0$ as $r\to0$. 
After $k$ steps, if $x_k \in U$ then the game moves to step $k+1$. Otherwise, if $x_k \in \partial U$ then the game ends and player II pays the payoff $g_{m_k}(x_k)$ to player I as they are at mode $m_k=\nu(t_k)$. Notice that the change of modes is determined solely by the Markov chain \eqref{markov}, and is not determined by the two players. In particular, $\nu(t_k)$ can take any value in $\{1,\dots,m\}$ with probability determined by \eqref{markov}. The expected payoff is 
\[
\E_{i}\big[g_{\nu(t_k)}(x_k)\big].
\]

A \textit{strategy} for a player is a way of choosing the players' next move as a function of all previous information (played moves, all known coin tosses and known states.)  
It is a map from the set of partially played games to moves (or in the case of a random strategy, a probability distribution on moves.) Usually, one would think of a good strategy as being Markovian, i.e., as a map from the current state to the next move. However, in some settings, it is also useful to allow more general strategies that take into account the history. 

We consider the \textit{value} which the players get. 
Of course player I wants to maximize the expected payoff, 
while player II wants to minimize it in this tug-of-war game.
Let $S_{I}$ and $S_{II}$ be the strategies of player I and player II, 
respectively, and then we define the cost functions by   
\begin{align*}
J_{i}^{\ep}(S_I,S_{II})(x)
:=
\left\{
\begin{array}{ll}
\E_{S_{I},S_{II}}\E_{i}\big[g_{\nu(t_k)}(x_k)\big] & 
\textrm{if the game terminates with probability one}, \\
-\infty & 
\textrm{otherwise},  
\end{array}
\right. 
\end{align*}
where $x$ and $i$ are the starting point and mode of the game.
The value of the game for player I is then  defined as
\[
u^{\ep, I}_{i}(x):=\sup_{S_I}\inf_{S_{II}}J^{\ep}_{i}(S_I,S_{II})(x). 
\]

Note intuitively that if $x_l=x$ is far away from the boundary, i.e., 
$\dist(x,\partial U)>\ep$, then 
$x_{l+1}:=x_l+\ep v$ for some unit vector $v \in \bS^{n-1}$. On the other hand,  
if $\dist(x,\partial U) \leq \ep$, then 
$x_{l+1}$ could be any point on $\partial U \cap \ol{B}(x,\ep)$. 
We can easily get the dynamic programming principle associated with 
the value function: 
\begin{equation}\label{DPP1}
u_i^{\ep,I}(x)=
\frac{1}{2} \left \{  
\max_{y\in\partial B(x,\ep)}\E_{i}\big[u_{\nu(\ep^2)}^{\ep,I}(y)\big]
 +\min_{y\in\partial B(x,\ep)}\E_{i}\big[u_{\nu(\ep^2)}^{\ep,I}(y)\big] \right\},   
\end{equation}
since the players use a fair coin.

Let us suppose that 
\begin{equation}\label{conv}
u_{i}^{\ep,I}\to u_i^{I} \quad \text{uniformly on} \ U \ \text{as} \ \ep\to0, 
\end{equation}
and prove that $(u_1^{I},\dots,u_m^{I})$ solves the system \eqref{INF} by a heuristic argument using the dynamic programming principle.
We write $u^{\ep}_{i}, u_i$ for $u^{\ep, I}_{i},u^I_i$ respectively by abuse of notations.

Fix $i\in\{1,\cdots,m\}$. 
Note that the evolution of the jump process which is given by \eqref{markov} 
is deterministically governed by an ordinary differential equation: 
\begin{equation}\label{ODE}
\begin{cases}
(\rho^{i}_{k})_t + \displaystyle \frac{1}{2}\sum_{j=1}^m c_{jk} \rho^{i}_j=0 \quad &\text{in}\ (0,\infty)  \\ 
\rho_{k}^{i}(0)=\del_k^i \quad &\text{for}\ k=1,\ldots,m,
\end{cases} 
\end{equation}
where $\del_k^i=1$ if $k=i$, and $\del_k^i=0$ otherwise for given $k \in \{1,\ldots,m\}$, 
and $\rho^i_k(s)=\mathbb P(\nu(s)=k \mid \nu(0)=i)$ for all $s\geq 0$.
It is straightforward to derive that $0 \le \rho_{k}^{i} \le 1$ for all $i,k$ and
$$
\sum_{k=1}^m \rho_{k}^{i}(s)=1, \quad \text{for} \ s \ge 0,
$$
and $\lim_{s\to \infty} \rho_{k}^{i}(s)=1/m$ for $k=1,\ldots,m$. 
On the other hand, the matrix $(c_{ij})_{i,j=1}^m$
has a simple eigenvalue $0$, and its other eigenvalues
$\lam_1,\ldots,\lam_{m-1}$ have
positive real parts. 
Hence $\rho_k^{i}$ can be written as
\begin{equation*} \label{cij-1}
\rho_k^{i}(s) = \frac{1}{m} + \sum_{l=1}^{m-1} a_{kl} e^{\lam_l s}
\end{equation*}
for some constants $a_{kl} \in \C$ for $1 \le k\le m,\  1 \le l \le m-1$.
By using $\rho_{k}^{i}$, we can rewrite \eqref{DPP1} in an explicit form as 
\begin{equation}\label{DPP2}
u_i^\ep(x)=
\frac{1}{2} \left \{  
\max_{y\in\partial B(x,\ep)}
\sum_{k=1}^{m}\rho_{k}^{i}(\ep^2)u^\ep_{k}(y)
+\min_{y\in\partial B(x,\ep)}\sum_{k=1}^{m}\rho_{k}^{i}(\ep^2)u^\ep_{k}(y) 
\right\}. 
\end{equation}

We only give here the formal calculation to derive the system of the partial differential equations which $(u_1,\dots,u_m)$ satisfies. 
Set 
\[v^\ep(x):=\sum_{l=1}^{m}\rho_{l}^{i}(\ep^2)u^\ep_{l}(x).\] 
The dynamic programing principle reads, in light of the Taylor expansion, 
if $|Dv^{\ep}(x)|\not=0$, then 
\begin{align*}
u_i^\ep(x)&\approx\frac{1}{2}\left( v^\ep\Big(x+\frac{\ep Dv^\ep(x)}{|Dv^\ep(x)|}\Big) 
+ v^\ep\Big(x-\frac{\ep Dv^\ep(x)}{|Dv^\ep(x)|}\Big)\right)\\
&=v^\ep(x) + \frac{1}{2}\ep^2 D^2v^\ep(x) \frac{Dv^\ep(x)}{|Dv^\ep(x)|} \cdot \frac{Dv^\ep(x)}{|Dv^\ep(x)|}+o(\ep^2). 
\end{align*}
Noting that 
\begin{align*}
&\lim_{\ep\to0}\frac{\rho_{i}^{i}(\ep^2)-1}{\ep^2}
=(\rho_{i}^{i})_{t}(0)=\sum_{j=1}^{m}\frac{c_{ij}}{2}(1-\del_{j}^{i}), \\
&\lim_{\ep\to0}\frac{\rho_{j}^{i}(\ep^2)-0}{\ep^2}
=(\rho_{j}^{i})_{t}(0)=-\frac{c_{ij}}{2} \quad \text{for} \ j\ne i, 
\end{align*}
in view of \eqref{gen-cij}, we get 
\[
\lim_{\ep\to0}\frac{u^\ep_i(x)-v^\ep(x)}{\ep^2}
=\frac{1}{2}\sum_{j=1}^{m}c_{ij}(u_{i}-u_{j})(x). 
\]
Therefore, divide the above relation by $\ep^2/2$ and let $\ep \to 0$ to achieve the conclusion. \medskip

We can prove the convergence \eqref{conv} by a similar argument 
to that of \cite[Theorem 1.3]{PSSW} and also 
make the above proof rigorous 
by using the notion of viscosity solutions. 
See \cite{PSSW, BEJ} for more details. 
Thus, we have 
\begin{thm}
We have $u_{i}^{\ep,I}$ converge uniformly on $U$ for $1\leq i \leq m$ 
as $\ep\to0$. 
Let the limit functions be $u_i^{I}$. 
Then $(u_1^{I},\dots,u_m^{I})$ is a viscosity solution of \eqref{INF}.
\end{thm}

Analogously, we define the value of the game for player II by 
\[
u^{\ep, II}_{i}(x):=\inf_{S_I}\sup_{S_{II}}J^{-}_{i}(S_I,S_{II})(x), 
\]
and we have 
\begin{thm}
We have $u_{i}^{\ep,II}$ converge uniformly on $U$ for $1\leq i \leq m$ 
as $\ep\to0$. 
Let the limit functions be $u_i^{II}$. 
Then $(u_1^{II},\dots,u_m^{II})$ is a viscosity solution of \eqref{INF}.
\end{thm}

\begin{rem}\label{MT23}
We also refer to \cite{BEJ, MT2,MT3} and  the references 
for other studies of 
weakly coupled systems.   
More precisely, in Mitake and Tran \cite{MT2, MT3}, we investigated some properties of asymptotic limits 
of solutions to the weakly coupled system for Hamilton--Jacobi equations. 
It is quite important in both sense, intuitively and analytically, to look at the insight of solutions by considering 
the dynamic programming of the optimal control of the system 
whose states are governed by ODEs \eqref{ODE}, 
subject to random changes in the dynamics. 
See Barron, Evans, and Jensen \cite{BEJ} for some related phenomena.
\end{rem}


\section{Uniqueness result}\label{sec:unique}

As described in Section \ref{sec:derivation}, we can prove the existence of 
viscosity solutions of \eqref{INF} by using the tug-of-war game argument. 
In this section we investigate the uniqueness result for \eqref{INF}. 
We follow the arguments of Barles and Busca \cite{BB}. 
Henceforth we only consider the simple system with two equations and 
we assume $c_{12}=c_{12}=1, c_{11}=c_{22}=-1$ for an easy explanation. 
The general case follows quite straightforwardly.

We recall the definition of viscosity solutions of the weakly coupled system of the infinity Laplace equations. 
For a $C^2$ function $\varphi$ defined in a neighborhood of $x\in U$, one sets
$$
\Del_\infty^+ \varphi(x):=
\begin{cases}
\Del_\infty \varphi(x) \qquad &\text{if} \ D\varphi(x) \neq 0,\\
\max\{ D^2 \varphi(x) v\cdot v \mid  v\in\bS^{n-1}\}  &\text{if} \ D\varphi(x)=0,
\end{cases}
$$
and
$$
\Del_\infty^- \varphi(x):=
\begin{cases}
\Del_\infty \varphi(x) \qquad &\text{if} \ D\varphi(x) \neq 0,\\
\min\{ D^2 \varphi(x) v\cdot v \mid  v\in\bS^{n-1}\}  &\text{if} \ D\varphi(x)=0. 
\end{cases}
$$

\begin{defn}
A pair $(u_1,u_2) \in \USC(U)^2$ is a viscosity subsolution of \eqref{INF} if for any $i \in \{1,2\}$ and any  test function $\varphi\in C^2(U)$ such that if $u_i-\varphi$ has a local maximum at $x_0 \in U$ then
$$
-\Del_\infty^+ \varphi(x_0)+u_i(x_0)-u_j(x_0) \leq 0, 
$$
where $j=3-i$.

A pair $(u_1,u_2) \in \LSC(U)^2$ is a viscosity supersolution of \eqref{INF} if for any $i \in \{1,2\}$ and any  test function $\varphi \in C^2(U)$ such that if $u_i-\varphi$ has a local minimum at $x_0 \in U$ then
$$
-\Del_\infty^- \varphi(x_0)+u_i(x_0)-u_j(x_0) \geq 0, 
$$
where $j=3-i$.

We say that $(u_1,u_2)$ is a viscosity solution of \eqref{INF} if $(u_1,u_2)$ is both a subsolution and a supersolution of \eqref{INF}. 
\end{defn}

We first give the Hopf Lemma, which is an essential tool to achieve uniqueness. 
We only state the results for supersolutions. The results for subsolutions are the same with obvious changes.

\begin{lem}[The Hopf Lemma]\label{lem:Hopf}
Let $V$ be an open set such that $\ol V\subset U$.
Assume that $(u_1,u_2)$ is a supersolution of \eqref{INF} and that there exists $x_0 \in \partial V$ such that 
\begin{equation*}
u_1(x_0)=\min_{i=1,2} \min_{x\in U} u_i(x) \quad \text{and} \quad u_1(x_0) < u_i(x) \quad \text{for} \ i =1,2, \ x \in V.
\end{equation*}
Assume further that $V$ satisfies the interior ball condition at $x_0$, namely, there exists an open ball $B \subset V$ with $x_0 \in \partial B$. Then
\begin{equation*}
\liminf_{s\to0}\frac{u_1(x_0-s\nu(x_0))-u_1(x_0)}{s}>0, 
\end{equation*}
where $\nu(x_0)$ is the outward normal vector to $\partial V$ at $x_0$.
\end{lem}

\begin{thm}[Strong Maximum Principle]\label{thm:SMP}
Assume that $U$ is connected, open, and bounded, and $(u_1,u_2)$ is a supersolution of \eqref{INF}.
Assume further that $\min_{i=1,2} \min_{U} u_i$ is attained at an interior point of $U$. Then $u_1=u_2 \equiv C$
for some constant $C$ in $U$.
\end{thm}
The proofs of Lemma \ref{lem:Hopf} and Theorem \ref{thm:SMP} are similar to those for the single equation. 
We give them for the sake of clarity in Appendix by using the arguments 
based on those in \cite{BB}.

We recall next the change of variable which were introduced 
in \cite{BB}. 
Assume that we have $U_i=\psi(u_i)$, where $\psi:\R \to \R$ is smooth and invertible. We could also write that $u_i=\varphi(U_i)$ for $\varphi=\psi^{-1}$. One could then compute that
\begin{align*}
&-\Del_\infty u_1+ u_1 - u_2\\
=\ &-\frac{\sum_{i,j}(U_1)_{x_i} (U_1)_{x_j}\left(\varphi''(U_1)(U_1)_{x_i} (U_1)_{x_j}+\varphi'(U_1)(U_1)_{x_i x_j}\right)}{|DU_1|^2}+\varphi(U_1)-\varphi(U_2)\\
= \ &-\varphi'(U_1) \Del_\infty U_1 -\varphi''(U_1) |DU_1|^2+ \varphi(U_1)-\varphi(U_2)\\
= \ &  \varphi' (U_1) \left(-\Del_\infty U_1 + \frac{-\varphi''(U_1)}{\varphi'(U_1)}|DU_1|^2 + \frac{\varphi(U_1)-\varphi(U_2)}{\varphi'(U_1)}  \right). 
\end{align*}

From the above computation, one could notice that for $|DU_1|>0$, the system is strictly monotone provided that
$$
\left(\frac{-\varphi''(s)}{\varphi'(s)}\right)'>0
\iff (\varphi''(s))^2 >\varphi'''(s) \varphi'(s).
$$

We will therefore select a family of $\{\varphi_\ep\}_{\ep>0}$ satisfying two properties:  
\begin{itemize}
\item[(i)]\, 
The function $\varphi_\ep$ is close to the identity function, $\varphi_\ep'>0$, and $\varphi_\ep'$ converges to $1$ locally uniformly in $\R$ as $\ep \to 0$; 
\item[(ii)]\, $\varphi_\ep''$ converges to $0$ locally uniformly in $\R$ as $\ep \to 0$, and $(\varphi_\ep''(s))^2 >\varphi_\ep '''(s)\varphi_\ep'(s)$ for all $s\in \R$.
\end{itemize}
As in \cite{BB}, we can find such functions $\varphi_\ep$ defined as
\begin{equation}\label{def:phi}
\varphi_\ep'(t)=\exp \left(\int_0^t \exp(-\ep^{-1}(s+\ep^{-1}))\,ds \right).  
\end{equation}

\begin{thm}[Comparison Principle]\label{thm:CP}
Assume that $(u_1,u_2)$, $(v_1,v_2)$ are, respectively, 
a bounded subsolution, and supersolution of \eqref{INF}, and that $u_i \leq v_i$ on $\partial U$ for $i=1,2$. Then $u_i \leq v_i$ in $\cU$ for $i=1,2$.
\end{thm}

Before presenting the proof of the comparison principle, we recall two general  important properties of a semi-convex function $\ul{w}$ and a semi-concave function $\ol{w}$ in $U$, which will be used below in the proof of Theorem \ref{thm:CP}. 
\smallskip
\begin{itemize}
\item[(DMP)] Both $\ul{w}$ and $\ol{w}$ are differentiable at any local maximum points of $\ul{w}-\ol{w}$.
\smallskip
\item[(PCG)] For $w=\ul{w}$ or $\ol{w}$, if $w$ is differentiable at $x_0 \in U$ and if $\{x_n\}\subset U$ is a sequence of differentiable points of $w$ such that $x_n \to x_0$, then $Dw(x_n) \to Dw(x_0)$.
\end{itemize}

\begin{proof}[Proof of Theorem {\rm\ref{thm:CP}}]
We argue by contradiction. 
Suppose that 
$$
c:=\max_{i=1,2} \max_{x\in U} (u_i(x)-v_i(x))>0.
$$
We present the proof in several steps.

{\it Step 1.} 
We replace $(u_1,u_2)$ by $(u_1-c/2,u_2-c/2)$. 
We may assume further that $u_i$ are semi-convex and $v_i$ are semi-concave for $i=1,2$ by using sup and inf convolutions and restricting the problem to a slightly smaller domain  if necessary. 

We now perform a perturbation of $u_i$ as follows. For $\al >  0$, set $U_\al:=\{x\in U\,:\, \dist(x,\partial U)>\al\}$ and for $h\in \R^n$ with $|h|$ sufficiently small, we define
$$
M(h):=\max_{i=1,2}\max_{x\in \ol{U}_{|h|}}(u_i(x+h)-v_i(x))=u_{i_h}(x_h+h)-v_{i_h}(x_h)
$$
for some $i_h\in \{1,2\}$ and $x_h \in \ol{U}_{|h|}$.
As $M(0)>0$, for $|h|$ small enough, we have $M(h)>0$ and the above maximum is the same if we take it over $\ol{U}_\al$ for any $\al>0$ sufficiently small and fixed.
In particular, $x_h \in U_{|h|}$.
Note that actually at $x_h$, $u_1(x_h+h)-v_1(x_h)=u_2(x_h+h)-v_2(x_h)$.

{\it Step 2.} We now proceed the proof by assuming the following additional assumption, which will be verified in Step 3. Assume that

\begin{itemize}
\item[(H)] there exists a sequence $\{h_n\} \to 0$ such that: At any maximum point $y\in U_{|h_n|}$ of $\max_{i=1,2}\max_{x\in \ol{U}_{|h_n|}}(u_i(x+h_n)-v_i(x))$, we have $Du_{i}(y+h_n)=Dv_{i}(y) \neq 0$ for all $i\in\{1,2\}$ and all $n\in \N$.
\end{itemize}
In light of (H) together with (PCG), we yield the existence of a positive constant $\del(n)>0$ so that $|Du_i(y+h_n)|=|Dv(y)|>\del(n)$ for all such $y$ described in (H) and for $i=1,2$. Note that $\del(n)$ could vanish as $n\to \infty$ but it does not matter our analysis here.

By abuse of notation, we write $M(h_n), U_{|h_n|},x_{h_n}$ as $M(n), U_n, x_n$, respectively. We now perform the changes of variables as, for $\ep>0$ sufficiently small,
$$
U_i^\ep=\psi_\ep(u_i), \quad V_i^\ep=\psi_\ep(v_i) \quad \text{for}\ i=1,2.
$$
It is clear to see that $U_i^\ep$ are semiconvex and $V_i^\ep$ are semiconcave, where $\psi_{\ep}:=(\varphi_{\ep})^{-1}$ and $\varphi_{\ep}$ is defined by \eqref{def:phi}. 
We get that $\max_{i=1,2} \max_{x\in U_n} (U_i^\ep(x+h_n)-V_i^\ep(x))$ is achieved at some point $x_\ep\in U_n$ and by passing a subsequence if necessary, $x_\ep \to x_n$ as $\ep\to0$. As $|Du_i(x_n+h_n)|=|Dv_i(x_n)| >\del(n)$, we deduce further that for $\ep$ sufficiently small, $|DU_i^\ep(x_\ep+h_n)|=|DV_i^\ep(x_\ep)| \ge \del(n)/2$.

Note that this is enough for the system with $(U_1^\ep,U_2^\ep)$, $(V_1^\ep, V_2^\ep)$ to be strictly monotone as discussed above. We can then get the contradiction. See the proof of \cite[Lemma 3.1]{BB} for a more detailed discussion.

{\it Step 3.} We finally verify that (H) holds. This is indeed a very important property and is correct in light of the Hopf Lemma (i.e., Lemma \ref{lem:Hopf}) 
and the strong maximum principle (i.e., Theorem \ref{thm:SMP}).

Were (H) false, there would exist, for each $h$ with $|h|$ small, $x_h \in U_{|h|}$ which is a maximum point of $\max_{i=1,2} \max_{x\in U_h} (u_i(x+h)-v_i(x))$  so that $Du_{i_h}(x_h+h)=Dv_{i_h}(x_h)=0$ for some $i_h\in \{1,2\}$.
As $u_i-v_i$ is semiconvex, $M(h)$ is hence semiconvex in a neighborhood of $0$.
Now for any $k$ close to $h$, one could compute that, in light of $Du_{i_h}(x_h+h)=0$,
\begin{equation*}
M(k)\geq u_{i_h}(x_h+k)-v_{i_h}(x_h) \geq u_{i_h}(x_h+h)-C|h-k|^2-v_{i_h}(x_h)=M(h)-C|h-k|^2.
\end{equation*} 
Thus, $0 \in \partial M(h)$ for all $h$ with $|h|$ small, 
where $\pl M$ denotes the subdifferential of $M$. 
This yields that $M(h)\equiv M(0)$ in a neighborhood of $0$.

Now pick $x_0 \in U$ to be a maximum point of $\max_{i=1,2} \max_{x\in U} (u_i(x)-v_i(x))$. We note that, for $h$ sufficiently small so that $x_0 \in U_{|h|}$, and for $i\in \{1,2\}$,
$$
u_i(x_0)-v_i(x_0)=M(0)=M(h) \geq u_i(x_0+h)-v_i(x_0).
$$
Thus, $x_0$ is a point of local maximum of $\max_{i=1,2} u_i$. In light of the strong maximum principle, this actually implies that $u_i$, $v_i$ are constants in $U$, which gives the desired result.
\end{proof}

\begin{cor}
Let $u_{i}^{I}$,  $u_{i}^{II}$ be the functions defined by the limit functions 
described in Section {\rm\ref{sec:derivation}}. 
Then we have $u_{i}^{I}=u_{i}^{II}$ on $\cU$. 
\end{cor}

\begin{rem}
We can prove the uniqueness by using the game interpretation 
which is similar to that in \cite{PSSW}. 
We show both arguments here since we want to explain that to consider 
the system \eqref{INF} is natural in both sense, i.e., 
game theoretical and PDE points of view.  
We also refer to Armstrong and Smart \cite{AS} for a simple proof of the uniqueness of the single infinity Laplace equation by using comparison with cones and intuition from \cite{PSSW}. 
\end{rem}

\section{Comparison with ``generalized cones'' for systems}\label{sec:cone}
For the single infinite Laplace equation 
\begin{equation}\label{eq:single}
-\Del_{\infty}u=0 \quad \text{in} \ U,
\end{equation}
Crandall, Evans and Gariepy \cite{CEG} realized that comparison with cones 
characterizes subsolutions and supersolutions of \eqref{eq:single}, and nowadays it is well-known that 
this plays important roles in the establishment of regularity results of solutions of \eqref{eq:single}. See \cite{Sa,ESa,ESm}. 
In this section, we derive ``generalized cones'' for systems and establish comparison with ``generalized cones".

We first present one way to find the class of particular  
solutions of \eqref{eq:single}, and 
that cones are solutions of \eqref{eq:single} everywhere except the vertices.  
Let us find radially symmetric solution  $u$ of \eqref{eq:single}, i.e.    
\[
u(x)=\eta(|x|), 
\]
where $\eta:[0,\infty)\to\R$ is some smooth function.
We calculate, for $x\neq 0$,
\begin{align*}
&Du(x)=\eta'(|x|)\cdot \frac{x}{|x|}, \\ 
&D^{2}u(x)=\eta''(|x|)\cdot \frac{x\otimes x}{|x|^2}
+\eta'(|x|)\cdot\Big(I-\frac{x\otimes x}{|x|^2}\Big)\frac{1}{|x|^2}.  
\end{align*}
Plug these into \eqref{eq:single} to get that
\[
-\eta''(r)=0,   
\]
which implies that $\eta(r)=ar+b$ for any $a,b\in\R$. 
From these calculations, we establish that the cones
\begin{equation}\label{cone-single}
u(x)=a|x-x_0|+b \quad \text{for any} \ x_0\in \R^n,\  \text{and} \ a,b\in\R 
\end{equation}
are solutions of \eqref{eq:single} in $U\setminus\{x_0\}$.

\subsection{``Generalized cones''}
Following the idea above, we first find particular solutions of \eqref{INF} 
in the form of cones' like. 
We consider $u_i$ radially symmetric of the form
\[
u_i(x)=\eta_i(|x|),
\]
where  $\eta_i:[0,\infty)\to\R$ are smooth functions for $i=1,2$.
Assume that $(u_1,u_2)$ is a solution of \eqref{INF} in $\R^n \setminus \{0\}$.
Then $(\eta_1,\eta_2)$ satisfies 
\begin{equation}\label{INF-special}
\begin{cases}
-\eta_1''+\eta_1-\eta_2=0 \quad &\text{in}\ (0,\infty),\\
-\eta_2''+\eta_2-\eta_1=0  &\text{in}\ (0,\infty). 
\end{cases}
\end{equation}
Solving this system of ordinary differential equations with arbitrary initial data at $0$,  we get that, for $s>0$,
\[
\begin{cases}
\eta_1(s)=C_1e^{\sqrt{2}s}+C_2e^{-\sqrt{2}s}+as+b,\\
\eta_2(s)=-C_1e^{\sqrt{2}s}-C_2e^{-\sqrt{2}s}+as+b,  
\end{cases}
\]
where $C_1,C_2, a, b$ are arbitrary constants. 

We can then easily check that the pair $(\psi_1,\psi_2)$ defined by 
\begin{equation}\label{cone-system}
\begin{cases}
\psi_1(x):=C_1e^{\sqrt{2}|x-x_0|}+C_2e^{-\sqrt{2}|x-x_0|}+a|x-x_0|+b,\\
\psi_2(x):=-C_1e^{\sqrt{2}|x-x_0|}-C_2e^{-\sqrt{2}|x-x_0|}+a|x-x_0|+b,  
\end{cases}
\end{equation}
is a solution of \eqref{INF} in $\R^n \setminus \{x_0\}$ 
for any $x_0\in\R^n, C_1,C_2, a, b\in\R$. 
We call $(\psi_1,\psi_2)$ a pair of \textit{``generalized cones"}. 

\begin{rem}\label{rem:entire} 
It is clear to see that $\psi_1$ is differentiable at $x_0$ if only if $\sqrt{2}(C_1-C_2)+a=0$,
and $\psi_2$ is differentiable at $x_0$ if only if $\sqrt{2}(C_2-C_1)+a=0$. 
Thus $\psi_1, \psi_2$ are both differentiable at $x_0$ if and only if $C_1=C_2$ and $a=0$.
In particular, $(e^{\sqrt{2}|x-x_0|}+e^{-\sqrt{2}|x-x_0|},-e^{\sqrt{2}|x-x_0|}-e^{-\sqrt{2}|x-x_0|})$ is a solution of 
\eqref{INF} in the whole $\R^n$. Notice that this is a highly nontrivial fact for weakly coupled systems because 
it is not the case for the single equation. 

\end{rem}

\subsection{Comparison with ``generalized cones''}\label{subsec:cone}

We introduce the notion of comparison with ``generalized cones" following  the single case. 
\begin{defn}[Comparison with ``Generalized Cones"]
{\rm(i)} 
A pair $(u_1,u_2)\in C(\cU)^2$  \textit{enjoys comparison with ``generalized cones" from above in $U$} if  
$(u_1,u_2)$ satisfies that 
for any $x_0\in U$ and $r>0$ such that $\ol{B}(x_0,r)\subset U$, 
\[\text{if}\  u_i\le \psi_i \  \text{on} \ \pl B(x_0,r)\cup\{x_0\} \ \text{for} \ i=1,2, \  
\text{then} \  
u_i\le \psi_i \ \text{on} \ \ol{B}(x_0,r) \ \text{for} \ i=1,2,\] 
for any choices of $C_1,C_2,a,b\in \R$.\\
{\rm(ii)} 
A pair $(u_1,u_2)\in C(\cU)^2$ \textit{enjoys comparison with ``generalized cones" from below in $U$} if  
$(u_1,u_2)$ satisfies that 
for any $x_0\in U$ and $r>0$ such that $\ol{B}(x_0,r)\subset U$, 
\[\text{if}\  u_i\ge \psi_i \  \text{on} \ \pl B(x_0,r)\cup\{x_0\} \ \text{for} \ i=1,2, \  
\text{then} \  
u_i\ge \psi_i \ \text{on} \ \ol{B}(x_0,r) \ \text{for} \ i=1,2,\] 
for any choices of $C_1,C_2,a,b\in \R$.
\end{defn}

In the case of the single equation \eqref{eq:single}, to characterize subsolutions by using comparison with cone, 
one could choose in \eqref{cone-single}
\[
a:=\frac{\max_{|y-x_0|=r}u(y)-u(x_0)}{r}, \quad
b:=u(x_0).
\]
For comparison with  ``generalized cones" for systems, we need to appropriately 
choose $C_1, C_2, a, b$ in \eqref{cone-system}. 
In order to do so, we introduce the following notations. 
For $x_0\in U$, $r>0$ such that $\ol{B}(x_0,r)\subset U$, we set 
\begin{align*}
&M_{i}(x_0,r):=\max_{|y-x_0|=r}u_i(y), \\
&C_{1}(x_0,r):=\frac{-(u_1(x_0)-u_2(x_0)) e^{-\sqrt{2}r}}{2(e^{\sqrt{2}r}-e^{-\sqrt{2}r})}+\frac{M_{1}(x_0,r)-M_{2}(x_0,r)}{2(e^{\sqrt{2}r}-e^{-\sqrt{2}r})},\\
&C_{2}(x_0,r):=\frac{(u_1(x_0)-u_2(x_0)) e^{\sqrt{2}r}}{2(e^{\sqrt{2}r}-e^{-\sqrt{2}r})}-\frac{M_{1}(x_0,r)-M_{2}(x_0,r)}{2(e^{\sqrt{2}r}-e^{-\sqrt{2}r})},\\
&a(x_0,r):=\frac{M_{1}(x_0,r)+M_{2}(x_0,r)-(u_1(x_0)+u_2(x_0))}{2r},\\
&b(x_0):=\frac{u_1(x_0)+u_2(x_0)}{2}. 
\end{align*}

\begin{thm}[Characterization of Subsolutions and Supersolution of \eqref{INF}]\label{thm:cone}\ \\
Let $(u_1,u_2)\in C(\cU)^{2}$. 
The pair $(u_1,u_2)$ is a viscosity subsolution {\rm(}resp., supersolution{\rm)} of \eqref{INF} if and only if 
$(u_1,u_2)$ satisfies comparison with  ``generalized cones" from above {\rm(}resp.,  below{\rm)}.
\end{thm}
\begin{proof}
We only prove this property for subsolutions. 
It is obvious to check that if $(u_1,u_2)$ is a viscosity subsolution of \eqref{INF},  then $(u_1,u_2)$ satisfies comparison with ``generalized cones" 
by the comparison principle (i.e., Theorem \ref{thm:CP}).

We assume now that $(u_1,u_2)$ satisfies comparison with ``generalized cones" from above, and we will  prove that $(u_1,u_2)$ is a viscosity subsolution of \eqref{INF}. 
Fix $x_0\in U$ and choose $r>0$ so that $\ol{B}(x_0,2r)\subset U$. 
Take a test function $\phi\in C^2(U)$ such that $u_1-\phi$ attains a strict  maximum at $x_0$ with $u_1(x_0)=\phi(x_0)$. 
Set $p:=D\phi(x_0)$. 

By subtracting a suitable constant from both $u_1$ and $u_2$, we can assume that $u_1(x_0)=d, u_2(x_0)=0$ without loss of generality. 
For any $y\in B(x_0,r)$, let $\psi_{i}[y,r](\cdot)$ be the ``generalized cones" defined by \eqref{cone-system} with $C_i=C_i(y,r)$, $a=a(y,r)$, $b=b(y)$.  
Let $\psi_{i}[x_0,r](\cdot)$ be the ``generalized cones" defined by \eqref{cone-system} with $C_i=C_i(x_0,r)$, $a=a(x_0,r)$, $b=b(x_0)$ for $i=1,2$ which are defined above. 
Then we can easily check 
\[
\psi_{i}[x_0,r](x_0)=u_i(x_0), \ \text{and} \  
u_i(x)\le\psi_{i}[x_0,r](x) \ \text{for} \ x\in\partial B(x_0,r) \ \text{and} \ i=1,2. 
\]
Hence, by the assumption, $u_i(x)\le\psi_{i}[y,r](x)$ for all $x\in B(y,r)$ and  $i=1,2$.

Let $y:=x_0-sp$ and choose $s>0$ sufficiently small so that $s|p|<r$. 
Set $\chi(r):=e^{\sqrt{2}r}-e^{-\sqrt{2}r}$. 
Noting that $u_1\le \phi$ on $U$, $u_1(x_0)=\phi(x_0)$, and $u_1(x_0) \leq \psi_1[y,r](x_0)$,
we get 
\begin{align*}
\phi(x_0)\le&\,  
\frac{\chi(r-s|p|)}{2\chi(r)}(\phi(y)-u_2(y))
+\frac{\chi(s|p|)}{2\chi(r)}(\phi(z_{r,s})-M_2(r,y))\\
&+\frac{s|p|}{2r}(\phi(z_{r,s})+M_2(r,y))+
\frac{r-s|p|}{2r}(\phi(y)+u_2(y)),  
\end{align*}
where $z_{r,s}\in \pl B(y,r)$ so that 
$\phi(z_{r,s})=\max_{|z-y|=r}\phi(z)$. 
We rewrite this as 
\begin{align*}
& \frac{(r-s|p|)}{2r}(\phi(x_0)-\phi(y))+\frac{s|p|}{2r}(\phi(x_0)-\phi(z_{r,s}))\\
&+\frac{1}{2}\Big(\phi(x_0)-\frac{\chi(r-s|p|)}{\chi(r)}\phi(y)\Big)
-\frac{\chi(s|p|)}{2\chi(r)}\phi(z_{r,s})\\
\le&
-\Big(\frac{\chi(r-s|p|)-\chi(r)}{2\chi(r)}+\frac{s|p|}{2r}\Big)u_2(y)
+\Big(\frac{s|p|}{2r}-\frac{\chi(s|p|)}{2\chi(r)}\Big)M_2(y,r).   
\end{align*}

We may assume $z_{r,s}\to z_{r}\in \pl B(x_0,r)$ as $s\to0$ by taking a subsequence if necessary. 
Divide by $s>0$ and let $s\to 0$ to get 
\begin{align*}
&|p|^2+|p|\Big\{\Big(\frac{1}{2r}+\frac{\sqrt{2}}{\chi(r)}\Big)(\phi(x_0)-\phi(z_r))
+\frac{-2\sqrt{2}+\chi'(r)}{2\chi(r)}\phi(x_0)\Big\}\\
\le&\, 
|p|\Big(\frac{1}{2r}-\frac{\sqrt{2}}{\chi(r)}\Big)M_{2}(x_0,r),
\end{align*}
where we used the fact that $u_2(x_0)=0$ here on the right hand side. 

Assume first $p\ne0$. Divide the above inequality by $|p|>0$ to yield that
\begin{align*}
&|p| +\left(\frac{1}{2r}+\frac{\sqrt{2}}{\chi(r)}\right)(\phi(x_0)-\phi(z_{r}))
+\frac{(e^{r/\sqrt{2}}-e^{-r/\sqrt{2}})\phi(x_0)}{\sqrt{2}(e^{r/\sqrt{2}}+e^{-r/\sqrt{2}})}\\
\le&\,
\left(\frac{1}{2r}-\frac{\sqrt{2}}{\chi(r)}\right)M_{2}(x_0,r).
\end{align*}
We have used here an elemental calculation  
\[
\frac{-2\sqrt{2}+\chi'(r)}{2\chi(r)}=
\frac{(e^{r/\sqrt{2}}-e^{-r/\sqrt{2}})}{\sqrt{2}(e^{r/\sqrt{2}}+e^{-r/\sqrt{2}})}.
\]
Note that 
\[
\frac{\sqrt{2}}{\chi(r)}=\frac{1}{2r}+O(r),  
\]
and $|M_{1r}(x_0,r)-d|+|M_{2r}(x_0,r)|=\om(r)$ for some modulus $\om$.
Thus,
\[
|p|-\frac{\phi(z_r)-\phi(x_0)}{r}+\frac{(e^{r/\sqrt{2}}-e^{-r/\sqrt{2}})\phi(x_0)}{\sqrt{2}(e^{r/\sqrt{2}}+e^{-r/\sqrt{2}})}\le r\om(r).
\]

By the Taylor theorem, 
\[
\phi(z_r)-\phi(x_0)=
p\cdot(z_{r}-x_0)+\frac{1}{2}D^{2}\phi(x_0)(z_r-x_0)\cdot(z_r-x_0)
+o(r^2).
\]
Also note that 
\[
\frac{z_{r}-x_0}{r} \to \frac{D\phi(x_0)}{|D\phi(x_0)|}
=\frac{p}{|p|} 
\quad \text{as} \ r\to0, 
\]
since $\phi(z_r)=\max_{|z-x_0|=r}\phi(z)$. 
Hence, 
\[
|p|-p\cdot \frac{z_{r}-x_0}{r}-
\frac{1}{2}D^{2}\phi(x_0)\frac{z_r-x_0}{r}\cdot(z_r-x_0)-\frac{(e^{r/\sqrt{2}}-e^{-r/\sqrt{2}})\phi(x_0)}{\sqrt{2}(e^{r/\sqrt{2}}+e^{-r/\sqrt{2}})}\le 
r\om(r).
\]
Noting that $r^{-1}(z_r-x_0)\in\bS^{n-1}$ which implies 
$|p|-p\cdot r^{-1}(z_{r}-x_0) \geq 0$, we can simplify further the above inequality as 
$$
-\frac{1}{2}D^{2}\phi(x_0)\frac{z_r-x_0}{r}\cdot(z_r-x_0)-\frac{(e^{r/\sqrt{2}}-e^{-r/\sqrt{2}})\phi(x_0)}{\sqrt{2}(e^{r/\sqrt{2}}+e^{-r/\sqrt{2}})}\le 
r\om(r).
$$
Dividing the above by $r>0$, and letting $r\to 0$, we get 
$$
-\frac{1}{2} \Del_\infty \phi(x_0)+\frac{d}{2} \leq 0.
$$
Therefore,
$$
-\Del_\infty\phi(x_0)+u_1(x_0)-u_2(x_0)=-\Del_\infty\phi(x_0)+d \leq 0.
$$

Consider now  the case $p=D\phi(x_0)=0$. 
Let $y:=x_0-sq$ for $q\in\bS^{n-1}$ and choose $0<s<r$. 
By performing similar computations to the above case, we get
\begin{align*}
-\frac{\phi(z_r)-\phi(x_0)}{r}+\frac{(e^{r/\sqrt{2}}-e^{-r/\sqrt{2}})\phi(x_0)}{\sqrt{2}(e^{r/\sqrt{2}}+e^{-r/\sqrt{2}})}\le r\om(r). 
\end{align*}
Thus, dividing the above by $r>0$, and letting $r\to 0$, and then taking the definition of the operator $\Del_{\infty}^{+}$ into account, we achieve that 
$-\Del_{\infty}^{+} \phi(x_0)+(u_1-u_2)(x_0)\le0$. 
\end{proof}

\section{Linearity of blow up limits}\label{sec:blowup}
We now use the generalized cones to deduce further properties of $(u_1,u_2)$.
Fix $x_0 \in U$ and choose $r>0$ satisfying $\ol{B}(x_0,r)\subset U$. 
We set 
\begin{align*}
&S_{i}^{+}(x_0,r):=\frac{M_{i}(x_0,r)-u_i(x_0)}{r}, \\
&S_{i}^{-}(x_0,r):=\frac{m_{i}(x_0,r)-u_i(x_0)}{r}
\end{align*}
for $i=1,2$, where $M_{i}(x_0,r)$ is defined in Section \ref{subsec:cone} 
and $m_{i}(x_0,r):=\min_{|y-x_0|=r}u_i(y)$. 

In the case of the single equation \eqref{eq:single}, 
as in \cite[Lemma 2.4]{CEG}, 
the associated functions $r\mapsto S^{+}(x_0,r)$ and $r\mapsto S^{-}(x_0,r)$ 
are nondecreasing and nonincreasing, respectively. 
We first point out an example to show that this is not the case for weakly coupled systems.
We thus have to be careful with the system setting here 
as the phenomena are very different with the single case.

\begin{ex}\label{ex1}
Let us consider \eqref{INF} with $U=B_1$, $g_1\equiv-1$ and $g_2\equiv1$ on $\partial B_1$. 
We can easily check that the pair of functions $(v_1,v_2)\in C^2(\ol{B}_1)$ defined by 
$$
v_1(x):=-\frac{e^{\sqrt{2}|x|}+e^{-\sqrt{2}|x|}}{e^{\sqrt{2}}+e^{-\sqrt{2}}}, \quad \text{and} \quad v_2:=-v_1
$$
is the unique solution, which is actually the classical solution. 
Then we can easily check that $M_1(0,r)$ is decreasing. 
Indeed, for any $r,s\in (0,1)$ with $s<r$,
$$
M_1(0,r)=
-\frac{e^{\sqrt{2}r}+e^{-\sqrt{2}r}}{e^{\sqrt{2}}+e^{-\sqrt{2}}}
<-\frac{e^{\sqrt{2}s}+e^{-\sqrt{2}s}}{e^{\sqrt{2}}+e^{-\sqrt{2}}}
=M_1(0,s). 
$$
Moreover,
$$
s\mapsto S_1^+(0,s)=-\frac{e^{\sqrt{2}s}+e^{-\sqrt{2}s}-2}{(e^{\sqrt{2}}+e^{-\sqrt{2}})s}
\quad \text{is decreasing.}
$$
\end{ex}
\smallskip

Let $(u_1,u_2)\in C(\cU)^{2}$ be the viscosity solution of \eqref{INF} hereinafter. 
We only have the monotonicity of $a_r$, which is the summation of $S_i^+$.
This is an easy and straightforward result coming from comparison with generalized cones. 
\begin{prop}\label{prop:increase}
Let $a_r$ be the function defined in Subsection {\rm\ref{subsec:cone}}. 
Then, $r \mapsto a_r(x_0,r)$ is nondecreasing for any $x_0\in U$.
\end{prop}

\begin{proof}
Let $\psi_{i}[x_0,r]$ be generalized cones defined in Subsection \ref{subsec:cone}. 
Then we have 
$$
u_i\le\psi_{i}[x_0,r] \quad \text{in} \ B(x_0,r), \ \text{for}\ i\in\{1,2\}.
$$ 
Plug $x\in \pl B(x_0,s)$  with $s\leq r$ in the above inequality and maximize over all such $x$  to deduce that 
\begin{align}
&M_{1}(x_0,s) \leq C_{1}(x_0,r) e^{\sqrt{2}s}+C_{2}(x_0,r) e^{-\sqrt{2}s}
+a(x_0,r) s + \frac{u_1(x_0)+u_2(x_0)}{2},\label{M-ineq}\\
&M_{2}(x_0,s) \leq -C_{1}(x_0,r) e^{\sqrt{2}s}-C_{2}(x_0,r) e^{-\sqrt{2}s}
+a(x_0,r) s + \frac{u_1(x_0)+u_2(x_0)}{2}.\notag
\end{align}
Summing up the two inequalities, we immediately get the conclusion.
\end{proof}

We have the following important inequalities, which play the role as a generalization of 
\cite[Lemma 2.4]{CEG}, and are important tools to analyze the system \eqref{INF}.   
\begin{lem} \label{lem:Ll}
We have, for $x_0\in U$, $0<s\le r\le \dist(x_0,\pl U)$, and $i\in\{1,2\}$,
\begin{align*}
&SC_{i}^{+}(x_0,s) \leq 
\frac{1}{2} \left(1+\frac{\xi(s)}{\xi(r)}\right)SC_{i}^{+}(x_0,r)
+\frac{1}{2} \left(1-\frac{\xi(s)}{\xi(r)}\right)SC_{j}^{+}(x_0,r),\\
&SC_{i}^{-}(x_0,s) \geq 
\frac{1}{2} \left(1+\frac{\xi(s)}{\xi(r)}\right)SC_{i}^{-}(x_0,r)
+\frac{1}{2} \left(1-\frac{\xi(s)}{\xi(r)}\right)SC_{j}^{-}(x_0,r),
\end{align*}
where $j=3-i$, $\xi(r):=(e^{\sqrt{2}r}-e^{-\sqrt{2}r})/r$, and 
\begin{equation}\label{def:SC}
SC_i^{\pm}(x_0,r):=S_i^{\pm}(x_0,r)+\frac{u_i(x_0)-u_j(x_0)}{2}\cdot\frac{1-e^{-\sqrt{2}r}}{r}. 
\end{equation}
\end{lem}

\begin{proof}
We only consider the case $i=1$ and prove the first inequality. 
As usual,
by subtracting a suitable constant from both $u_1$ and $u_2$, we can 
assume $u_1(x_0)=d, u_2(x_0)=0$. 
We rewrite \eqref{M-ineq} as 
\begin{align*}
\frac{M_1(x_0,s)-d}{s}
\le&\,  
\frac{d(e^{\sqrt{2}(r-s)}-e^{-\sqrt{2}(r-s)})}{2(e^{\sqrt{2}r}-e^{-\sqrt{2}r})s}
+ 
\frac{(M_1(x_0,r)-M_2(x_0,r))(e^{\sqrt{2}s}-e^{-\sqrt{2}s})}{2(e^{\sqrt{2}r}-e^{-\sqrt{2}r})s}\\
&+\frac{M_1(x_0,r)-d+M_{2}(x_0,r)}{2r}-\frac{d}{2s}\\
=&\,  
\frac{d}{2}\Big\{
\frac{e^{\sqrt{2}(r-s)}-e^{-\sqrt{2}(r-s)}}{2(e^{\sqrt{2}r}-e^{-\sqrt{2}r})s}
+\frac{\xi(s)}{r\xi(r)}-\frac{1}{2s}\Big\}\\
&+\frac{1}{2} \left(1+\frac{\xi(s)}{\xi(r)}\right)S_{1}^{+}(x_0,r) + \frac{1}{2} \left(1-\frac{\xi(s)}{\xi(r)}\right)S_{2}^{+}(x_0,r).
\end{align*}
Noting that 
\[
\frac{e^{\sqrt{2}(r-s)}-e^{-\sqrt{2}(r-s)}}{2(e^{\sqrt{2}r}-e^{-\sqrt{2}r})s}
=-\frac{e^{-\sqrt{2}r}}{2r}\cdot\frac{\xi(s)}{\xi(r)}+\frac{e^{-\sqrt{2}s}}{2s}, 
\]
we get 
\begin{align*}
S_{1}^{+}(x_0,s) 
\leq& \, 
\frac{d}{2}\left ( \frac{1-e^{-\sqrt{2}r}}{r}\cdot\frac{\xi(s)}{\xi(r)} - \frac{1-e^{-\sqrt{2}s}}{s} \right )\nonumber\\
&+\frac{1}{2} \left(1+\frac{\xi(s)}{\xi(r)}\right)S_{1}^{+}(x_0,r) + \frac{1}{2} \left(1-\frac{\xi(s)}{\xi(r)}\right)S_{2}^{+}(x_0,r),
\end{align*}
which implies the conclusion. 
\end{proof}

\begin{lem}\label{lem:Lip}
The viscosity solution $(u_1, u_2)$ of \eqref{INF} are locally Lipschitz continuous in $U$.
\end{lem}

\begin{proof}
In view of Lemma \ref{lem:Ll}, we have 
$$
SC_{1}^{+}(x_0,s)  \leq \max_{i=1,2} SC_{i}^{+}(x_0,r), 
\quad 
SC_{1}^{-}(x_0,s)\geq \min_{i=1,2} SC_{i}^{-}(x_0,r)
$$
for $x_0\in U$ and $0<s<r$ such that $\ol{B}(x_0,r)\subset U$, 
which implies  
\[
S_1^{+}(x_0,s)\le C^{+}(r) \quad \text{and} \quad  
S_1^{-}(x_0,s)\ge C^{-}(r) \ 
\]
for some $C^{\pm}(r)\in\R$ and any $0<s\le r$. 
Therefore, $u_1$ is Lipschitz continuous at $x_0$.  
Furthermore, one could compute more explicitly that,
\[
|Du_1(x_0)| \leq \max\left\{\max_{i=1,2}|S_{i}^+(x_0,r)|, \max_{i=1,2} |S_{i}^-(x_0,r)|\right\}+\sqrt{2}|u_1(x_0)-u_2(x_0)|.
\qedhere
\]
\end{proof}

\begin{lem}
There exists the limits of $S_{i}^{+}(x_0,r)$ as $r\to0$, which is denoted by $S_i^+(x_0)$ for $i=1,2$. 
\end{lem}
\begin{proof}
By Lemma \ref{lem:Ll}, 
\[
\limsup_{s\to0}SC_{1}^{+}(x_0,s) \leq 
\frac{1}{2} \left(1+\frac{2\sqrt{2}}{\xi(r)}\right)SC_{1}^{+}(x_0,r)
+ \frac{1}{2} \left(1-\frac{2\sqrt{2}}{\xi(r)}\right)SC_{2}^{+}(x_0,r). 
\]
Then, noting that $1\pm 2\sqrt{2}/\xi(r)>0$ and taking the liminf in $r$ of the right hand side of the above, 
we get 
\[
\limsup_{s\to0}SC_{1}^{+}(x_0,s) \leq \liminf_{r\to0}SC_{1}^{+}(x_0,r),
\]
which is the conclusion. 
\end{proof}

For each $R>0$, set 
\[L_{i}(y,R):=\sup_{0<s<R}S_i^+(y,s).
\]

\begin{lem}\label{lem:S-m}
For $y\in U$, $0<\ep<r$ such that $\ol{B}(y,r)\subset U$, we have\smallskip\\
{\rm(i)} \ 
$\displaystyle 
-L_{i}(y,r)+O(r)+\frac{O(\ep)}{r}\le 
\min_{|x-y|=\ep}\frac{u_i(x)-u_i(y)}{\ep}$, \\
{\rm(ii)} \ 
$S_i^+(y)=-S_i^-(y)$. 
\end{lem}
\begin{proof}
Suppose that $y=0$. Pick $0<\ep<r<R$ such that $R>r+\ep$ and $\ol{B}(0,R)\subset U$.

In light of comparison with generalized cones, if $0<\ep:=|x|<r$, then 
\begin{align*}
&u_{1}(0) \leq C_{1}(x,r) e^{\sqrt{2}\ep}+C_{2}(x,r) e^{-\sqrt{2}\ep}+a(x,r) \ep + \frac{u_1(x)+u_2(x)}{2},
\end{align*}
where the functions $C_{i}, a$ are give in Section \ref{subsec:cone}.

Using $e^{\pm\sqrt{2}\ep}=1\pm\sqrt{2}\ep+O(\ep^2)$, we get 
\begin{align*}
&C_{1}(x,r) e^{\sqrt{2}\ep}+C_{2}(x,r) e^{-\sqrt{2}\ep}\\
\le&\, 
\frac{u_1(x)-u_2(x)}{2}
+
\Big[
-\frac{\sqrt{2}(e^{\sqrt{2}r}+e^{-\sqrt{2}r})(u_1(x)-u_2(x))}{2(e^{\sqrt{2}r}-e^{-\sqrt{2}r})}
+\frac{\sqrt{2}(M_{1}(x,r)-M_{2}(x,r))}{e^{\sqrt{2}r}-e^{-\sqrt{2}r}}
\Big]\cdot\ep\\
&+O(\ep^2). 
\end{align*}
Note that
\[
M_{1}(x,r)\le u_1(x)+L_{1}(x,r)r
\le u_1(0)+L_1(0,\ep)\ep+L_{1}(x,r)r. 
\]  

We combine all the inequalities above to yield that
\begin{align*}
&\frac{u_1(x)-u_1(0)}{\ep}\cdot\frac{2r-\ep}{2r}+O(\ep)\\
\ge&\, 
\Big[
\frac{\sqrt{2}(e^{\sqrt{2}r}+e^{-\sqrt{2}r})(u_1(x)-u_2(x))}{2(e^{\sqrt{2}r}-e^{-\sqrt{2}r})}
-\frac{\sqrt{2}(M_1(x,r)-M_{2}(x,r))}{e^{\sqrt{2}r}-e^{-\sqrt{2}r}}
\Big]\\
&-\frac{L_{1}(0,\ep)\ep+L_1(x,r)r+M_{2}(x,r)-u_2(x)}{2r}\\
=&\,
-\frac{M_1(x,r)-u_1(x)}{2r}-\frac{L_{1}(x,r)}{2}+\frac{O(\ep)}{r}+O(r)\\
\ge&\,
-L_{1}(x,r)+\frac{O(\ep)}{r}+O(r), 
\end{align*}
which implies the conclusion of (i).

Send $\ep \to 0$ and $r\to 0$ in this order to deduce that
$$
S_i^-(0) \geq - S_i^+(0).
$$
Changing the roles of $(u_1,u_2)$ to $(-u_1,-u_2)$ to yield the result.
\end{proof}

\begin{rem}\label{rem:ball-sphere}
Recall that our situation here is very different compared to the single equation case. 
As seen in Example \ref{ex1}, 
maximum and minimum values on the spheres of solutions to \eqref{INF}
are widely different from those on the balls, which is not the case in the single equation. 
See \cite[Lemma 4.1 (a)]{C}. 
Therefore, we cannot expect for instance that 
a subsolution has a nondecreasing property of $S^{+}(y,\cdot)$ 
defined in \cite{C}. See \cite[Lemma 4.1 (d)]{C}. 
Moreover, $r\mapsto \max_{\ol{B}_r(0)}u_1$ is concave which is critically  different from \cite[Lemma 4.1 (e)]{C}.
Of course, this is not the case for single equation
as $u_{+}(x)=|x|$ is a subsolution but $u_{-}(x)=-|x|$ is not a subsolution.

However, we are able to achieve a sort of \textit{monotonicity} result in Lemma \ref{lem:Ll} for $SC_i^{\pm}$ given by \eqref{def:SC} in a coupled way. One could think of $2^{-1}\left(1\pm \xi(s)/\xi(r)\right)$ as the weights corresponding with the coupling terms. This is a sort of intuitive phenomenon appearing in the weakly coupled system context. See \cite{MT2, MT3} for similar phenomenon regarding the weights in the study of weakly coupled system of Hamilton--Jacobi equations. 
\end{rem}
\smallskip

Now, take $x_0 \in U$ and $R>0$ such that $\ol{B}(x_0,R) \subset U$.
For each $r>0$ sufficiently small, set
$$
v_{i}^{r}(x):=\frac{u_i(x_0+rx)-u_i(x_0)}{r}, \quad \text{for}\ |x| \leq \frac{R}{r}, \ i=1,2.
$$
Clearly $\{v_{i}^{r}\}$ is precompact in $C(B(0,R))$.
Thus for any sequence $\{r_j\}_{j\in\N}$ with $r_j \to 0$ as $j\to\infty$, 
we can pass to a subsequence if necessary and get 
$v_{i}^{r_j} \to v_i\in \Lip(\R^n)$ locally uniformly in $\R^n$ as $j\to\infty$. 
We call $v_i$ a blow-up limit of $u_i$. 
We now prove that all of blow-up limits $v_i$ are affine. 
Notice that $(v_1,v_2)$ here really depends on the subsequence we take. In general, a pair $(v_1,v_2)$ of blow-up limits depends on the choice of subsequences and it might not be unique. 

Let us recall the literature on regularity results for the single infinity Laplace equation here. 
Note first that in all of these papers, 
the result on affine blow-up limits  \cite{CEG} plays an important role.
Savin \cite{Sa} showed that this blow-up limit is unique and achieved $C^1$ regularity for solutions in  case $n=2$. Evans and Savin \cite{ESa} then established $C^{1,\al}$ regularity for solutions in this setting. The proofs in \cite{Sa, ESa} depend highly on the geometry of the $2$-dimensional space and cannot be extended to the case with $n\geq 3$.
Recently, Evans and Smart \cite{ESm} used the nonlinear adjoint method to prove that this blow-up limit is unique, which yields the differentiability everywhere of solutions for all $n\geq 2$. The questions on $C^1$ and $C^{1,\al}$ regularity, however, are still open for $n\geq 3$.

\begin{lem}\label{lem:sameLip}
We have 
\[
\|Dv_i\|_{L^\infty(\R^n)}=S_i^+(x_0)=-S_i^-(x_0) \quad \text{for any} \ x_0\in\R^n, \ 
\text{and} \ i=1,2.
\]
\end{lem}

\begin{proof}
Assume $x_0=0$ for simplicity. It is enough to show that, for every $R_0>0$,
\begin{equation}\label{eq:sL}
A:=
\limsup_{r \to 0} \sup_{|x|,\,|y|\leq R_0}\frac{v_1^{r}(y)-v_1^{r}(x)}{|y-x|}
=\limsup_{r \to 0} \sup_{|x|,\,|y|\leq R_0}\frac{u_1(ry)-u_1(rx)}{r|y-x|} 
\leq S_1^+(0).
\end{equation}
Fix $r_0>0$. For $r>0$ sufficiently small, we have $rR_0<r_0/2$, and thus $rx,ry \in B_{r_0/2}$ for all $x,y \in B_{R_0}$.

In light of Lemma \ref{lem:Ll}, we have 
\begin{align}
&SC_1^+(rx,r|x-y|)\nonumber\\ 
\leq&\, \frac{1}{2}\left(1+\frac{\xi(r|x-y|)}{\xi(r_0)}\right)SC_1^+(rx,r_0)+ \frac{1}{2}\left(1-\frac{\xi(r|x-y|)}{\xi(r_0)}\right)SC_2^+(rx,r_0). 
\label{eq:SLkey}
\end{align}
We use the definition of $SC_1^+$ to estimate first the left hand side of the above
$$
SC_1^+(rx,r|x-y|)\geq \frac{u_1(ry)-u_1(rx)}{r|x-y|}+\frac{u_1(rx)-u_2(rx)}{2}\frac{1-e^{-\sqrt{2}r|x-y|}}{r|x-y|}.
$$
On the other hand, noting that 
\[
\max_{|z-rx|=r_0}\frac{u_1(z)-u_1(0)}{r_0}
\le 
\begin{cases}
\displaystyle
\frac{r_0+r|x|}{r_0}\sup_{r_0-r|x|\leq s \leq r_0+r|x|}S_1^+(0,s)
&\text{if}\ S_1^+(0,s)\ge0\\
\displaystyle
\frac{r_0-r|x|}{r_0}\sup_{r_0-r|x|\leq s \leq r_0+r|x|}S_1^+(0,s)
&\text{if}\ S_1^+(0,s)<0, 
\end{cases}
\]
we get 
\begin{align*}
SC_1^+(rx,r_0)
=&\max_{|z-rx|=r_0}\frac{u_1(z)-u_1(rx)}{r_0}+\frac{u_1(rx)-u_2(rx)}{2}\frac{1-e^{-\sqrt{2}r_0}}{r_0}\\
=&\max_{|z-rx|=r_0}\frac{u_1(z)-u_1(0)}{r_0}+\frac{u_1(0)-u_1(rx)}{r_0}+\frac{u_1(rx)-u_2(rx)}{2}\cdot\frac{1-e^{-\sqrt{2}r_0}}{r_0}\\
\leq &\,
S_{1}^{+}(0,r_0)+\om(r)+\frac{u_1(0)-u_1(rx)}{r_0}+\frac{u_1(rx)-u_2(rx)}{2}\cdot\frac{1-e^{-\sqrt{2}r_0}}{r_0}
\end{align*}
for a modulus $\om\in C([0,\infty))$ with $\om(0)=0$. 
We can achieve similar estimate for $SC_2^+(rx,r_0)$.
We let $r \to 0$ in \eqref{eq:SLkey} and use the above estimates to yield that
\begin{multline}\label{eq:SLfin}
A+\frac{u_1(0)-u_2(0)}{2}\sqrt{2} 
\leq 
\frac{1}{2}\left(1+\frac{2\sqrt{2}}{\xi(r_0)}\right)\cdot\left(S_1^+(0,r_0)
+\frac{u_1(0)-u_2(0)}{2}\cdot\frac{1-e^{-\sqrt{2}r_0}}{r_0}\right)\\
+\frac{1}{2}\left(1-\frac{2\sqrt{2}}{\xi(r_0)}\right)\cdot\left(S_2^+(0,r_0)+\frac{u_2(0)-u_1(0)}{2}\cdot\frac{1-e^{-\sqrt{2}r_0}}{r_0}\right).
\end{multline}
Letting $r_0\to 0$ in \eqref{eq:SLfin} to get the desired result.
\end{proof}

We are now ready to prove one of the main results.
\begin{thm}\label{thm:blow-up}
All of blow up limits of $u_i$ are affine for $i=1,2$. 
\end{thm}
We give the proof which is based on \cite[Proposition 7.1]{C} here for confirmation.  
We also give a simple proof of \cite[Lemma 7.3]{C}. 
\begin{lem}[{\rm\cite[Lemma 7.3]{C}}]\label{lem:crandall}
Let $v$ be any Lipschitz continuous function on $\R^n$ such that the Lipschitz constant is $1$.
Assume further that there exist $x_0\in \R^n$ and $p \in \bS^{n-1}$ such that
\[
v(x_0+tp)=v(x_0)+t \quad \text{for all} \ t\in\R. 
\]
Then $v$ is linear, and moreover 
\[
v(x_0+x)=v(x_0)+p\cdot x \quad \text{for all} \ x\in\R^n.  
\]
\end{lem}
\begin{proof}
We may assume $x_0=0$, $v(0)=0$ and $p=(1,0,\dots,0)$ without loss of generality. 
Setting $w(x_1,\dots,x_n):=v(x_1+a,\dots,x_n)-a$ for a fixed $a\in\R$, 
we only need to prove $w(0,x_2,\dots,x_n)=0$.  

By the assumption of $v$ in the statement, we have 
\[
t-w(0,x_2,\dots,x_n)
\leq
|w(0,x_2,\dots,x_n)-w(t,0,\dots,0)|\le\sqrt{t^2+x_2^2+\cdots+x_n^2} 
\]
for all $t\in\R$, which implies $w(0,x_2,\dots,x_n)\ge0$. 
Suppose $c:=w(0,x_2,\dots,x_n)>0$. 
Raising the last inequality of the above to the second power, we get 
$$
-2ct+c^2\le x_2^2+\cdots+x_n^2,
$$
 which is a contradiction 
for $t<-1$ with $|t|$ large. 
\end{proof}

\begin{proof}[Proof of Theorem {\rm\ref{thm:blow-up}}]
We may assume $x_0=0$. 
Let $v_1\in\Lip(\R^n)$ be a blow-up limit of $u_1$, and then we can take 
a sequence $r_j\to0$ as $j\to\infty$ such that 
$v_1^{r_j}\to v_1$ as $j\to\infty$. 

Fix $t>0$ and choose $x_j\in\pl B(0,1)$ so that 
\begin{align*}
&v_1^{r_j}(tx_j):=
\max_{|x|=t}v_1^{r_j}(x)\\
=&\,
\frac{\max_{|x|=t}u_1(r_jx)-u_1(0)}{r_j}
=
t\cdot\frac{\max_{|x|=1}u_1(tr_jx)-u_1(0)}{tr_j}
=
tS_{1}^{+}(0,r_j).  
\end{align*}
We may assume $x_j\to x_{t}^{+}\in\pl B(0,1)$ as $j\to\infty$ 
by replacing it by a subsequence if necessary. 
In view of Lemma \ref{lem:sameLip}, sending $j\to\infty$ yields 
\[
v_1(tx_{t}^{+})=t\|Dv_1\|_{\Li(\R^n)}. 
\]
Similarly, 
\[
v_1(tx_{t}^{-})=-t\|Dv_1\|_{\Li(\R^n)} \quad\text{for some} \ 
x_{t}^{-}\in\pl B(0,1).  
\]
Therefore, 
\[
2t\|Dv_1\|_{\Li(\R^n)}=v_1(tx_{t}^{+})-v_1(tx_{t}^{-})
\le t\|Dv_1\|_{\Li(\R^n)}|x_{t}^{+}-x_{t}^{-}|. 
\]
We only need to consider the case $\|Dv_1\|_{\Li(\R^n)}\not=0$, 
since if not the case, then $v_1\equiv0$ and we are done. 
If $\|Dv_1\|_{\Li(\R^n)}\not=0$, then we have $x_{t}^{+}=-x_{t}^{-}$, 
which implies $v_1$ is linear on the segment $[tx_{t}^{-},tx_{t}^{+}]$, 
\[
\frac{v(tx_{t}^{+})-v(tx_{t}^{-})}{|tx_{t}^{+}-tx_{t}^{-}|}
=\|Dv_1\|_{\Li(\R^n)}. 
\]  
We may assume again $x_{t}^{+}\to p$ as $t\to\infty$ for some 
$p\in \bS^{n-1}$ by taking a subsequence.  
Thus, we get 
\[
v_1(tp)=t\|Dv_1\|_{\Li(\R^n)} \quad\text{for all} \ t\in\R. 
\]
In view of Lemma \ref{lem:crandall}, we get the conclusion. 
\end{proof}

\section{Further interpretation}\label{sec:interpretation}
We notice that by multiplying 
the two equations of \eqref{INF} with $|Du_i|^2$ for $i=1,2$ respectively, 
we can rewrite it as
\begin{equation}\label{INF-re}
\begin{cases}
-\sum_{i,j}(u_1)_{x_i} (u_1)_{x_j} (u_1)_{x_i x_j}+|Du_1|^2(u_1-u_2)=0 \qquad &\text{in}\ U\\
-\sum_{i,j}(u_2)_{x_i} (u_2)_{x_j} (u_2)_{x_i x_j} +|Du_2|^2(u_2-u_1)=0 \qquad &\text{in} \ U
\end{cases}
\end{equation}
We get from the first equation that $Du_1\cdot \left(D^2 u_1 D u_1-(u_1-u_2)Du_1\right)=0$. This means that the vector $Du_1$ is perpendicular to $D^2 u_1 Du_1-(u_1-u_2)Du_1$. It is also important to note that, the ``generalized cones'' $(\psi_1,\psi_2)$ defined in Section \ref{sec:cone} 
satisfy,
$$
D^2\psi_1 D\psi_1 - (\psi_1-\psi_2)D\psi_1=D^2\psi_2 D\psi_2 - (\psi_2-\psi_1)D\psi_2=0 \quad \text{in} \ \R^n\setminus\{x_0\}.
$$ 
It is clear that when $u_1=u_2$ then we are back to the case of single equation. However, when $u_1\neq u_2$, then the coupling  terms do play an important role in this relation and should be taken into account in the analysis.

We now provide some further heuristic interpretation following \cite{ESm}.
Let us now consider the formal linearized operator 
$(L_1[v_1,v_2],L_2[v_1,v_2])$ of \eqref{INF-re} which is given by
$$
\begin{cases}
L_1[v_1,v_2]:=-\sum_{i,j}(u_1)_{x_i} (u_1)_{x_j} (v_1)_{x_i x_j}- 2(D^2 u_1 Du_1-(u_1-u_2)Du_1)\cdot Dv_1\\
\hspace*{2.2cm}+|Du_1|^2(v_1-v_2),\\
L_2[v_1,v_2]:=-\sum_{i,j}(u_2)_{x_i} (u_2)_{x_j} (v_2)_{x_i x_j}- 2(D^2 u_2 Du_2-(u_2-u_1)Du_2)\cdot Dv_2\\
\hspace*{2.2cm}+|Du_2|^2(v_2-v_1).
\end{cases}
$$
We can interpret this linearized system as a weakly coupled system of linear parabolic equations. Firstly, $(u_1)_{x_i} (u_1)_{x_j} (v_1)_{x_i x_j}$ corresponds to the diffusion on the direction $Du_1$, which could be thought of as the ``space'' direction of the parabolic equation. Secondly, the drift term $(D^2 u_1 Du_1-(u_1-u_2)Du_1)\cdot Dv_1$ in the direction $(D^2 u_1 Du_1-(u_1-u_2)Du_1)$ perpendicular to $Du_1$ plays the role of ``time'' derivative of the parabolic equation. Intuitively, we can think of this system as a weakly coupled system of one-dimensional linear heat equations
$$
\begin{cases}
(v_1)_t-(v_1)_{xx}+v_1-v_2=0\\
(v_2)_t-(v_2)_{xx}+v_2-v_1=0.
\end{cases}
$$
Of course in this setting, either the diffusion or the drift can be degenerate. For example, the drift vanishes in the case of the generalized cones $(\psi_1,\psi_2)$ and we thus can think of this case as the stationary/degenerate elliptic case. On the other hand, if $|Du_1|=0$, then it corresponds to a ``singular case'' where everything vanishes in $L_1[v_1,v_2]$. We can think of this as the case at ``time-like'' infinity, which means that singularities are not reachable after a finite time. Therefore, at least at the heuristic level, regularity for solutions is expected.
We notice however that this is just a heuristic justification.
In this system, it is much more complicated because the ``time" and ``space" directions of 
the two equations in the linearized system are completely different.
Furthermore, the involvement of the coupling terms with the gradients in the ``time" directions
will make the analysis much more complicated. 


\section{Appendix}
\begin{proof}[Proof of Lemma {\rm\ref{lem:Hopf}}]
We note first that $u_1-0$ has a minimum at $x_0$, hence
$$
0+u_1(x_0)-u_2(x_0) \geq 0,
$$
which implies that $u_2(x_0)=u_1(x_0)=\min_{i=1,2} \min_{x\in U} u_i(x)$. By subtracting a constant if necessary, assume that $u_1(x_0)=u_2(x_0)=0$.

By the interior ball condition, there exists $R>0$ and $\ol{x}\in V$ 
such that $B(0,R) \subset V$ and $x_0 \in \partial B(\ol{x},R)$. 
We may assume $\ol{x}=0$ without loss of generality. 
We now build a specific subsolution of \eqref{INF}, which is a classical way to 
prove the strong maximum principle. 
For $\al>0$ to be chosen, set
$$
w(x):=e^{-\al |x|^2} - e^{-\al R^2} \quad \text{for} \ x \in A_R:=\{x \,:\, R/2 <|x| < R\}.
$$

We can easily see that $(w,w)$ is a strict subsolution of \eqref{INF} in $A_R$ for $\al$ sufficiently large by explicit computation. 
Indeed, for $x\in A_R$,
\begin{equation}\label{w-w-super}
-\Del_\infty w + w - w = e^{-\al |x|^2} (2 \al - 4 \al^2 |x|^2) \leq e^{-\al|x|^2} (2 \al - \al^2 R^2)<0
\end{equation}
for $\al>2 R^{-2}$.

We claim next that $u_i \geq w$ in $A_R$ for $i=1,2$ for $\al$ sufficiently large. Note that as $u_i>0$ in $V$, we can choose $\al$ very large so that $w(x)<u_i$ for $x \in \partial B(0,R/2)$ for $i=1,2$. Besides, $w \equiv 0 \leq u_i(x)$ on $\partial B(0,R)$. 
We argue by contradiction. Suppose that 
$$
\min_{i=1,2} \min_{A_R} (u_i-w)<0.
$$
Then, in view of the above observations, there exists $x_1 \in A_R$ and $k\in\{1,2\}$ so that
\begin{equation}\label{w-w-con}
(u_k-w)(x_1)=\min_{i=1,2} \min_{A_R} (u_i-w)<0.
\end{equation}
This implies in particular that $u_k(x_1) \leq u_l(x_1)$ for $l=3-k$. 
Noting that $Dw(x_1)\not=0$, 
by the definition of viscosity supersolutions,
\begin{equation*}
0 \leq-\Del_\infty^- w(x_1)+ u_k(x_1)-u_l(x_1) \leq -\Del_\infty^- w(x_1)= -\Del_\infty w(x_1),
\end{equation*}
which contradicts \eqref{w-w-super}. 

Thus $w \le u_1$ in $A_R$ and $w(x_0)=u_1(x_0)=0$, 
which implies $u_1-w$ has a minimum at $x_0$. 
Furthermore,
\[
\liminf_{s\to0}\frac{u_1(x_0-s\nu(x_0))-u_1(x_0)}{s}
\ge 
\liminf_{s\to0}\frac{w(x_0-s\nu(x_0))-w(x_0)}{s}
>0.
\]
The proof is complete.
\end{proof}

\begin{proof}[Proof of Theorem {\rm\ref{thm:SMP}}]
Set $m=\min_{i=1,2} \min_{U} u_i$ and  $W:=\{x \in U \,:\, \min_{i=1,2} u_i(x)=m\}$. We first note that for $x\in W$, we actually have $u_1(x)=u_2(x)=m$ as already observed in the above proof.
We  observe next that we could assume $u_i$ are semi-concave for $i=1,2$. If not, one could always replace $u_i$ by its inf-convolution, which is defined as, for some $\beta>0$ sufficiently small,
\begin{equation*}
u_i^\beta(x)=\inf_{y \in U} \left( u_i(y)+\frac{|y-x|^2}{\beta^2}\right).
\end{equation*}
Notice that $u_i^\beta(x)=u_i(x)=m$ for $x\in W$ and $i=1,2$, and $(u_1^\beta,u_2^\beta)$ is a supersolution of \eqref{INF}.
Also in light of (DMP), $u_i^\beta$ are differentiable at $x\in W$ and $Du_i^\beta(x)=0$.

Assume $u_i$ are semi-concave henceforth.
Define $V:=\{x\in U\,:\, \min_{i=1,2} u_i(x)>m\}$.
Now suppose that $V \neq \emptyset$, then we could choose a point $y\in V$ such that $\dist(y,W)<\dist(y,\partial U)$, and let $B(y,R)$ be the largest ball with center $y$ lying in $V$. Then there exists $x_0 \in W \cap \partial B(y,R)$. 
In light of the Hopf Lemma, 
\[
\liminf_{s\to0}\frac{u_1(x_0-s\nu(x_0))-u_1(x_0)}{s}
>0,
\] 
which gives the contradiction as $Du_1(x_0)=0$.
\end{proof}


\end{document}